\documentclass[a4paper,11pt]{amsart}

\usepackage{latexsym}
\usepackage{amssymb,amsmath}
\usepackage[pdftex]{graphicx}
\usepackage{url}

\newtheorem{theorem}{Theorem}

\newtheorem{lemma}[theorem]{Lemma}

\theoremstyle{definition}
\newtheorem{definition}[theorem]{Definition}

\renewcommand{\epsilon}{\varepsilon}

\newcounter{thmlistcnt}
	{\setcounter{thmlistcnt}{0}%
	\begin{list}{\emph{(\roman{thmlistcnt})}}{%
		\usecounter{thmlistcnt}%
		\setlength{\topsep}{0pt}%
		\setlength{\leftmargin}{0pt}%
		\setlength{\itemsep}{0pt}%
		\setlength{\itemindent}{17pt}}%
	}%
	{\end{list}}%

\newcounter{defnlistcnt}
\newenvironment{defnlist}%
	{\setcounter{defnlistcnt}{0}%
	\begin{list}{(\arabic{defnlistcnt})}{%
		\usecounter{defnlistcnt}%
		\setlength{\labelwidth}{18pt}%
		\setlength{\topsep}{-1pt}%
		\setlength{\leftmargin}{36pt}%
		\setlength{\itemsep}{1pt}%
		\setlength{\itemindent}{0pt}}%
	}%
	{\end{list}}%

\begin{document}

\begin{abstract}
This paper solves a pursuit--evasion problem in which a prince must find a princess who is constrained
to move on each day from one vertex of a finite graph to another. Unlike the related
and much studied `Cops and Robbers Game', the prince has no knowledge of the position of the princess;
he may, however, visit any single room he wishes on each day. We characterize the graphs for
which the prince has a winning strategy, and determine, for each such graph, the minimum number
of days the prince requires to guarantee to find the princess.
\end{abstract}


\title[Finding a princess in a palace]{Finding a princess in a palace:\\ A pursuit--evasion problem}

\author{John R. Britnell}
\address{Heilbronn Institute for Mathematical Research,
Department of Mathematics, University of Bristol, University Walk, Bristol BS8 1TW, United Kingdom}
\email{j.r.britnell@bristol.ac.uk}

\author{Mark Wildon}
\address{Department of Mathematics, Royal Holloway, University of London, Egham, Surrey TW20 0EX, United Kingdom}
\email{mark.wildon@rhul.ac.uk}

\date{2 April 2012}

\maketitle
\thispagestyle{empty}

\section{Introduction}

A princess
has set a visiting prince
the following challenge.
She will spend each day
in one of the rooms of her palace, and each night she will move into an adjacent room.
The prince may, at noon each day, demand admittance to one
room. If, within a finite
specified
number of days, he finds the princess,
she will agree to marry him. Otherwise, he must leave disappointed.
In this note we characterize the palaces in which the prince has
a winning strategy. We also determine, for each such palace, the minimum
number of days the prince requires to guarantee to find the princess.

The particular case of a palace consisting of $17$ rooms in
a row, and a prince allowed $30$ days to search for the princess,
was posed by Christian Blatter as a problem on
MathOverflow~\cite{BlatterPrincess}. A large number of related searching
problems have been studied in the literature. The closest to our
problem is
the game considered by Parsons in \cite{Parsons}, in which
a team of searchers constrained to move between adjacent vertices
in a graph (representing a cave network) must locate a similarly constrained lost
caver. Besides the constraint on the searchers, this problem also differs in an
important respect
from ours in that capture occurs on the edges, rather than the vertices, of the graph.

Two more related searching games, also with constrained searchers, are the `Hunter-Rabbit Game'
considered in \cite{IslerEtAl}, in which each side has intermittent knowledge
of the other's location, and the much-studied `Cops and Robbers Game' introduced
in \cite{NowakowskiWinkler}, in which both sides have perfect knowledge; see
\cite{AlonMehrabian} for some recent work in this area. We refer the
reader to \cite{FominThilikos} for a comprehensive bibliography of
searching problems.

Throughout the paper we shall assume an adversarial model, in which the
princess knows in advance which sequence of rooms the prince
will visit, and avoids him if she can. Since we are concerned only with the existence and length of guaranteed winning strategies for the prince, this model
is clearly legitimate.

We shall identify a palace with a finite connected simple graph, whose
vertices correspond to rooms, and two vertices are joined
by an edge if they correspond to adjacent rooms.
Let $T$ denote the tree with $10$ vertices
shown in Figure 1 above. (This tree appears as an example of a graph
in which the rabbit may avoid the hunter  in \cite[page 27]{IslerEtAl}.)

\smallskip
\begin{figure}
\begin{center}
\includegraphics{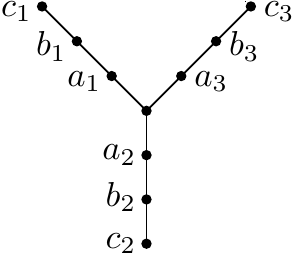}
\end{center}
\caption{The minimal palace $T$ in which the princess can avoid the prince.}
\end{figure}

\begin{theorem}
Let $G$ be a graph representing a palace.
If $G$ is a tree
not containing any subgraph isomorphic to $T$,
then the prince can guarantee to find the princess.
In any other case,
the princess can avoid the prince indefinitely.
\end{theorem}

\begin{proof}
It is clear that if
 $G$ contains a cycle then the princess
can avoid the prince while staying at the vertices on this cycle, since from any vertex
in the cycle there are two vertices to which she can proceed,
of which the prince can visit only one.

Suppose that $G$ contains $T$ as a subgraph. This part of the palace consists of a central vertex and three branches, which we number $1$, $2$ and $3$, each
with three vertices labelled as above.
We shall show that the princess can avoid the prince by following a strategy in
which on each even day she is either at the central vertex, or one of vertices $b_1$, $b_2$ or $b_3$.
We may assume that the prince visits one of these vertices on each even day.
The princess determines
her moves according to the following principles. If
on day $d$
the princess is at

\begin{itemize}
\item[(1)] vertex $c_i$, then she (necessarily) goes to $b_i$ on day $d+1$;

\item[(2)] vertex $b_i$, then she goes to $a_i$ on day $d+1$,
unless the prince is due
to visit~$a_i$ on day $d+1$,
in which case she goes to $c_i$;

\item[(3)] vertex $a_i$, then she goes to the central vertex on day $d+1$,
unless the
prince is due to visit the central vertex on day $d+1$,
in which case she goes to~$b_i$;

\item[(4)] the central vertex, then she goes to $a_i$, where branch $i$
is neither the branch the prince is due to visit on day $d+1$,
nor the branch $j$
in which he will next visit vertices $a_j$ and $b_j$ on two successive days, if such
a branch exists.
\end{itemize}

To see that this strategy allows the princess to avoid the prince indefinitely, we need only check that she never finds herself at a vertex $c_i$ when the prince is due to visit
$b_i$ on the next day. Since her visits to $c_i$ always coincide with the prince visiting $a_i$, this necessitates his visiting $a_i$ and $b_i$ on successive days. But that implies that he must have visited $a_j$ and $b_j$ in some other branch $j$ on successive days, in the time since the princess last visited the central vertex. And that is impossible, since her rules of movement would then have taken the princess to $a_i$ on the day the prince visited $a_j$, and back to the central vertex on the day he visited $b_j$.

We have established that if $G$ has a subgraph isomorphic to $T$ then the prince has no winning strategy. Conversely,
if $G$ is a tree
with no subgraph isomorphic to $T$ then the prince has a remarkably simple
strategy that will guarantee to find the princess.
Let~$P$ be a path of maximal length $\ell$ in $G$
and let
the vertices in~$P$ be numbered from~$0$ up to $\ell$ in the order they appear in $P$.
It will be useful to take a black-and-white
colouring of the vertices of $G$ that induces a bipartition of~$G$; for definiteness we shall
suppose that vertex $0$ is white.

Any non-leaf vertex in $G$ is either in $P$ or adjacent to a vertex in $P$.
Therefore the prince can begin at vertex $1$ on day $1$ and take a
walk to vertex $\ell-1$ that visits all non-leaf vertices \emph{en route},
never leaves $P$ for two consecutive days, never visits a vertex in $G\setminus P$ more than once, and never visits a lower-numbered
vertex in $P$ after a higher-numbered one. We call such a walk a \emph{linear search}.
For example, in the graph shown in Figure~2 below,
the prince might visit vertices in the order
$1$, $2$, $2a$, $2$, $3$, $3a$, $3$, $4$, $5$, $5a$, $5$, $5a'$, $5$, $6$.

\begin{figure}[h]
\begin{center}
\includegraphics{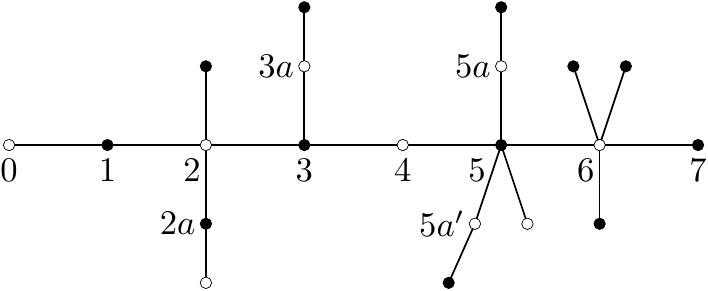}
\end{center}
\caption{A solvable palace with maximal path of length $7$.}
\end{figure}

We claim that if on day $1$ the princess is at a black vertex, then
the prince will find her in the course of his linear search. This is easily
proved by induction, observing that if the prince has not found the princess
when he visits vertex~$j$ for the final time, then the princess is on this
day at a vertex of the same colour as $j$, whose nearest vertex in $P$ is
 $j+1$ or higher. If when the prince visits vertex
$\ell-1$ the princess still eludes him, then he knows that on this
day the princess is at a vertex of the opposite colour to vertex $\ell-1$. The prince may therefore be certain of finding the princess if he
now performs a linear search of the opposite parity to the first: for instance, he can simply reverse his walk, beginning at vertex $\ell-1$ on the following day,
and ending at vertex~$1$.
\end{proof}

\begin{definition}
A strategy for the prince consisting of a linear search followed immediately by another linear search of the opposite parity will be called a \emph{linear strategy}.
\end{definition}

Let us say that a graph is \emph{solvable} if the prince has a strategy
that guarantees to find the princess.
The existence of a linear strategy
shows that any solvable palace can be solved
by a prince constrained so that the vertices he visits on consecutive days
are either adjacent or the same. Furthermore, he needs to be permitted to remain at the same vertex only once.
(It is obviously necessary in any solution
 that the prince should change his black--white parity at least once, since otherwise he can never find a princess who
starts with the opposite parity to himself.)

In the problem as posed, the prince is free to visit any vertex on any day, and it seems somewhat remarkable that this extra freedom is never required.
Equally surprising is that a linear strategy is optimal,
in that it uses the minimum number of days required to be certain of finding the princess.
The proof of this fact occupies the remainder of this paper.

\begin{definition}{\ }
\begin{defnlist}
\item We shall say that a leaf vertex in the graph of a palace is \emph{removable}
if it is adjacent to a vertex of degree at least $3$.
\item We say that a palace is \emph{reduced} is its graph
has no removable leaves.
\item If $G$ is a graph, we define
its \emph{reduction} $G^-$ to be the
graph obtained by repeatedly removing removable
leaves from the graph of $G$, one at a time, until no removable leaves remain. (It is clear that $G^-$ is well-defined up to graph isomorphism.)
\end{defnlist}
\end{definition}

For example, the reduction of the palace shown
in Figure~2 is obtained by deleting the leaf attached to
each of vertices $2$ and $5$, and any three of the four leaves attached to vertex $6$.

The next lemma shows that $G$ is solvable if and
only if $G^-$ is solvable, and that when either is solvable,
the optimal solutions to $G$ and~$G^-$
require the same number of days.

\begin{lemma}\label{lemma:reduction}
Let $G$ be the graph of a palace and let $H$ be obtained by removing
a removable leaf from $G$. Then
$G$ can be solved in $d$ days if and only if $H$ can be
solved in $d$ days.
\end{lemma}

\begin{proof}
Fix a numbering of the vertices of $G$  and assign to each
vertex in $H$ the corresponding number in $G$.
Let vertex $z$ be the leaf removed from $G$ to make $H$.
It suffices to show that if the prince has a winning strategy
in~$H$, in which he visits vertex $a_i$ on day $i$, for $1 \le i \le n$, then
the sequence of visits $(a_1,\ldots,a_n)$ also wins for the prince in~$G$.
If not, then there is a day~$d$ and a walk
\[ (w_1, \ldots, w_{d-2}, w_{d-1}, z, w_{d+1}, w_{d+2}, \ldots, w_n) \]
for the princess in $G$
such that $w_r \not= a_r$ for all $r \not= d$.
Since $z$ is a leaf in $G$, there is a vertex $y$ such that
$w_{d-1} = w_{d+1} = y$. Since the leaf $z$ is removable,
vertex $y$ has at least two neighbours
in $H$. Hence
there exists $x \not= z$ such that $x$ is adjacent to $y$ in $H$
and $x\not= a_{d-1}$. Now
\[ (w_1, \ldots, w_{d-2}, y, x, y, w_{d+2}, \ldots, w_n) \]
is a walk for the princess that defeats the prince's strategy on $H$,
which contradicts the hypothesis that his strategy is winning on $H$.
\end{proof}

Our second main result is as follows.
\begin{theorem}\label{thm:secondmain}
Let $G$ be the graph of a solvable palace and suppose that
the reduction $G^-$ of ~$G$ has exactly
$m$ vertices, where $m \ge 3$.
Then $G$ can be solved in $2m-4$ days by a linear strategy,
and $G$ cannot be solved in fewer days by any strategy.
\end{theorem}

For example, a palace whose graph is a path is certainly reduced.
Theorem~\ref{thm:secondmain} therefore implies that $30$ days are the minimum that will suffice
for the $17$-room palace in Blatter's version of the problem~\cite{BlatterPrincess}.

The hypothesis $m \ge 3$ rules out only the trivial cases where the original graph
$G$ has either one or two vertices; in these cases $G$
may be solved in one or two days respectively, and it is clear that no fewer days suffice.

We first establish that Theorem~\ref{thm:secondmain} holds for a particular class of reduced palaces.

\begin{lemma}\label{lemma:bound}
Let $k \ge 2$
and let $S$ be the star graph consisting of a central vertex $x$ together with $k$ disjoint
 paths
of length $2$ beginning at $x$. Let $a_1, \ldots, a_k$ be the vertices
adjacent to $x$ and let $b_1, \ldots, b_k$ be the vertices (other than $x$) adjacent
to $a_1,\ldots, a_k$, respectively. In any winning strategy for the prince,
the prince must visit vertex $x$ on $2k-2$ days. For each $i$,
there are two days on which he visits a vertex
in the set $B_i=\{a_i,b_i\}$.
\end{lemma}

\begin{proof}
Take a black-and-white colouring of $S$, with the central vertex and the vertices $b_i$ coloured white, and the
vertices $a_i$ coloured black.
We may suppose that the princess is constrained to visit white vertices on even days.
Her strategy is a variation on the strategy used in the palace $T$  in Figure~$1$. If on day $d$
the princess is at
\begin{itemize}
\item[(1)] vertex $b_i$ then she (necessarily) goes to $a_i$ on day $d+1$;
\item[(2)] vertex $a_i$ then she goes to the central vertex unless the prince is due to visit
the central vertex
on day $d+1$,
in which case she goes to $b_i$;
\item[(3)] the central vertex, then she goes to $a_i$, where $B_i$ is
whichever branch is the last that the prince is due to visit in the days
after $d$.

\end{itemize}
It is clear that if the prince is to find a princess who is following this strategy, then he must do so in one of the
vertices $a_i$, on a day when she has just proceeded there from $b_i$. Hence, by the rules of movement of the princess, he must have visited all of the other branches $B_j$
for $j\neq i$ in the time since the princess last visited the central vertex. But the princess will always visit the central vertex on an even day, unless the prince is due to visit it that day. It follows that the prince must make at least $k-1$ visits to the central vertex on consecutive even days. The situation in which the princess is constrained to visit white vertices on odd days is similar, and so the prince must visit the centre at least $2k-2$ times overall to defeat a princess with no parity constraint.

Finally it is obvious that the prince must visit each branch $B_i$ on at least two days, for otherwise
the princess may avoid him by alternating between the two vertices in a single
branch.
\end{proof}

We are now ready to prove Theorem~\ref{thm:secondmain}.
\begin{proof}
By Lemma \ref{lemma:reduction} we may assume that $G$ is reduced, and so $G$ itself has~$m$ vertices. By Theorem~1, $G$ is a tree. If $m\in\{3,4\}$ then $G$ is a path, and it is easy to check that a linear strategy gives a solution of length $2m-4$, which is clearly optimal in either case. So we assume that $m>4$.

Write $B$ for the set of leaf vertices, $A$ for the set of vertices which are adjacent to leaves, and~$Q$ for the set of vertices which are in neither $A$ nor~$B$.
(We shall see shortly that~$Q$ is the vertex set of a path in~$G$, but this is not required
until the final stage of the proof.)

Since $G$ is reduced, each vertex in $A$ has degree exactly $2$. It is clear that exactly one of its neighbours is in $B$; we thus have $|A|=|B|$, and since $A$,~$B$ and~$Q$ are pairwise disjoint, it follows that $2|A|+|Q|=m$. Two vertices of~$A$ cannot be adjacent since $m>4$, and so each must have a neighbour in~$Q$.

If $v\in Q$ then the ball of radius $2$ about $v$ is a star graph of the type described in Lemma \ref{lemma:bound}. It follows from the lemma that each such vertex must be visited at least
$2d(v)-2$ times in any winning strategy, where $d(v)$ is the degree of the vertex $v$. On the other hand, if $a\in A$ then $a$ has a neighbour $b\in B$ and another neighbour $v$ in $Q$. Now $\{a,b\}$ forms a branch of the star graph centred on $v$; Lemma \ref{lemma:bound} tells us that any winning strategy must include two visits to $\{a,b\}$. We therefore have a lower bound $L$ on the length of a winning strategy, given by
\[ L=2|A| + \sum_{v\in Q} (2d(v)-2). \]

Consider the edges of $G$. Since it is a tree, there are $m-1$ of them. Of these, there are $|A|$
between $A$ and $B$, and a further $|A|$
between $Q$ and $A$;
the rest connect elements of $Q$. It follows that the sum of vertex degrees of elements of $Q$ is
\[
\sum_{v\in Q} d(v) = 2(m-1) - 3|A|.
\]
We therefore have
\[
L=4(m-1) - 4|A| -2|Q| = 4(m-1) - 2m = 2m-4,
\]
as required.

It remains to be shown that a linear strategy succeeds on day $2m-4$ or earlier.
Let $P$ be a path of maximal length $\ell$ in $G$. Then it follows from Theorem 1
that the $\ell-4$ vertices of $P$ which are neither
leaves nor adjacent to leaves, are precisely the vertices of $Q$. In following a linear strategy, the prince
starts at an element of $A$, visits every element of $A$ exactly twice, never visits $B$, and visits each vertex $v\in Q$ precisely
$2(d(v)-1)$ times. It follows that this strategy achieves the lower bound $L$, as claimed.
\end{proof}

We end by stressing the fact that if the prince follows a linear strategy
then he never visits a leaf vertex.
We leave it as an exercise for the
reader to show that this property is
shared by all optimal strategies.
The proof of Theorem~\ref{thm:secondmain} now shows that
the multiset of vertices visited by the prince in an optimal
solution of a graph $G$ is uniquely determined by $G$. This is
consistent with a stronger conjecture: that the only optimal strategies are linear strategies.

\def\cprime{$'$} \def\Dbar{\leavevmode\lower.6ex\hbox to 0pt{\hskip-.23ex
 \accent"16\hss}D}

\end{document}